\documentclass{amsart}
\usepackage{latexsym,amssymb,amsmath,amsfonts,amsthm, enumerate, tikz, color,float,enumitem, theoremref, dsfont, mathrsfs, pgfplots, subfigure, theoremref, color}
\usepackage{comment}
\usepackage{mathtools}

\pgfplotsset{compat=1.16}

\mathtoolsset{showonlyrefs}

\newtheorem{theorem}{Theorem}

\newtheorem{lemma}[theorem]{Lemma}

\newtheorem{corollary}[theorem]{Corollary}
 
\newcommand{\f}{\frac}
\renewcommand{\P}{\mathbf P }

\newcommand{\0}{\mathbf 0}
\newcommand{\E}{\mathbf{E}}

\newcommand{\T}{\mathbb{T}}

\newcommand{\ind}[1]{\mathbf{1}{\{ #1 \}}}
\newcommand{\HOX}[1]{\marginpar{\footnotesize #1}}
\newcommand{\RR}{\mathfrak{R}}
\newcommand{\R}{\mathcal R}
\newcommand{\B}{\mathcal B}

\newcommand{\vlam}{{\vec\lambda}}
\newcommand{\vrho}{{\vec\rho}}

\title{Distance-dependent chase-escape on trees}

\author[Hernandez-Torres]{Sarai Hernandez-Torres}
\email{saraiht@im.unam.mx}
\author[Junge]{Matthew Junge}
\email{Matthew.Junge@baruch.cuny.edu}
\author[Ray]{Naina Ray}
\author[Ray]{Nidhi Ray}

\thanks{Hernandez-Torres was supported by ISF grant 1692/17. Junge was partially supported by NSF Grant \#2115936.}

\begin{document}

\maketitle

\begin{abstract}
    We give a necessary and sufficient condition for species coexistence in a parasite-host growth process on infinite $d$-ary trees. The novelty of this work is that the spreading and death rates for hosts depend on the distance to the nearest parasite. 
\end{abstract}

\section{Introduction}

Chase-escape is an interacting particle system inspired by certain parasite-host dynamics \cite{rand1995invasion,keeling1995ecology, metz2000geometry}. The same dynamics have been reinterpreted in a variety of applications: predator-prey systems, rumor scotching, infection spread, and malware repair in a device network \cite{bordenave2014extinction, rumor, de2015process, gilbert, bernstein2022chase}. Red (host) particles occupy and spread to adjacent sites of a graph according to exponential clocks while facing the threat of blue (parasite) particles. When a host is infected by a parasite, the host perishes. The site where this occurs is occupied by the parasite for all time thereafter. 

The authors of \cite{ced} introduced a variant called \emph{chase-escape with death} in which red particles die at a given rate, independently of the spread of parasites. They studied this process on infinite $d$-ary trees and characterized the phase behavior. Much of the analysis in \cite{ced} relied on a novel connection to weighted Catalan numbers. We deepen this connection by generalizing to the setting in which the host spreading and death rates depend on the distance to the nearest parasite.

We begin by defining \emph{generalized chase-escape} on a $d$-ary tree. Fix $d \geq 1$ and let $\T_d$ denote the infinite, rooted $d$-ary tree in which every vertex has $d$ children. Denote the root by $\0$ and let $\mathcal T_d$ be $\T_d$ augmented with an additional vertex $\mathfrak b$ attached to $\0$. Vertices of $\mathcal T_d$ are in one of four states $\{w,b,r,\dagger\}$. State $w$ is a ``white'' unoccupied site, state $b$ is a ``blue'' site occupied by a parasite, state $r$ is a ``red'' site occupied by a host, and state $\dagger$ is a site containing a ``dead'' host.  Given vertices $u,v \in \mathcal T_d$, define $|u-v|$ to be the number of edges on the unique shortest length path connecting $u$ and $v$. We write $\sigma(u)$ to be set of vertices on the shortest path connecting $\mathfrak b$ to $u$. 

The dynamics of \emph{distance-dependent chase-escape} are as follows. Vertex $\mathfrak b$ is initially in state $b$. The root is initially in state $r$. All other vertices begin in state $w$. Adjacent sites in states $(b,r)$ transition to $(b,b)$ according to a rate 1 Poisson process i.e., the time for each event is an exponential random variable with mean 1. To specify the red spreading and death rates we take vectors
\begin{align} 
\vlam &= (\lambda_1,\lambda_2,\hdots) \text { and }
\vrho = (\rho_1, \rho_2 ,\hdots )
\end{align}
of nonnegative real numbers.
Given a vertex $u \in \mathcal T$, define the distance to the nearest vertex in state $b$ as
\begin{align}
\ell(u) &= \min \{|u-v| \colon  v \in \sigma(u)  \text{ and $v$ is in state $b$} \}.
\end{align}
Note that $\ell$ depends on the current configuration of the tree.
%\HOX{Might need to add that $\ell$ also depends on the entire configuration, so make it $\ell(u,\omega_t)$.}
Adjacent vertices $(u,v)$ with $u$ in state $r$ and $v$ in state $w$ have $v$ transition to state $r$ according to a Poisson process with intensity $\lambda_{\ell(u)}$ i.e., after exponentially distributed times with mean $1/\lambda_{\ell(u)}$. Meanwhile, a vertex $u$ in state $r$ transitions to state $\dagger$ according to a Poisson process with intensity $\rho_{\ell(u)}$. Unless the dependence is important to highlight, we will typically write $\P(A)$ rather than $\P_{\vlam, \vrho}(A)$ for events $A$ pertaining to distance-dependent chase-escape.

In words, blue may only spread to red sites and does so at rate 1. Red may only spread to white sites and does so at a rate that depends on the distance to the nearest blue site. Similarly, the rate that a red particle dies depends its distance to the nearest blue particle.  The chase-escape with death process from \cite{ced} is the special case of our model with $\lambda_i \equiv \lambda$ and $\rho_i \equiv \rho$ for all $i \geq 1$ with $\lambda, \rho > 0$.

%Since the parameter space $(\lambda,\rho)$ considered in \cite{ced} is two-dimensional, it was natural to study the critical value 
%$$\rho_c(\lambda) = $$

%\subsection{Statement of Results} 

 We are interested in the persistence of both species. Let $\B$ equal the sites that are at some time colored blue. Since blue may only occupy sites that were at some time red, we say that \emph{coexistence} occurs if $\P(|\B|=\infty) >0$. We say that \emph{expected coexistence} occurs if the weaker condition $\E[|\B|] = \infty$ occurs. 
%We say that \emph{expected coexistence} occurs if $E|\B| = \infty $. 
  %\emph{escape} occurs if $\P(B) = 0< \P(R)$, and \emph{extinction} occurs if $\P(R) =0$. These cases are exhaustive since the parasitic dependence of blue on red ensures that $B \subseteq R$. 
  The work \cite{ced} provided a necessary and sufficient condition for coexistence to occur and described the behavior at the phase transition in detail. For the sake of a simpler, less technical generalization of a main idea from \cite{ced}, we focus on expected coexistence (rather than coexistence). 
  %We leave it to future work to describe the phase transition and what happens at criticality for coexistence. 
  
  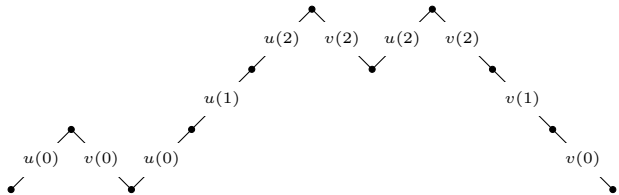
\begin{figure}
	\begin{tikzpicture}[scale = .8]
			%\draw (0,3) -- (0,0) -- (10,0);
			\draw (0,0) -- (1,1) node [midway,fill=white] {\tiny $u(0)$} -- (2,0)  node [midway,fill=white] {\tiny $v(0)$} -- (3,1) node [midway,fill=white] {\tiny $u(0)$} -- (4,2) node [midway,fill=white] {\tiny $u(1)$}-- (5,3) node [midway,fill=white] {\tiny $u(2)$} -- (6,2) node [midway,fill=white] {\tiny $v(2)$} --  (7,3) node [midway,fill=white] {\tiny $u(2)$} -- (8,2) node [midway,fill=white] {\tiny $v(2)$} -- (9,1) node [midway,fill=white] {\tiny $v(1)$} -- (10,0) node [midway,fill=white] {\tiny $v(0)$};
			
	\draw[fill] (0,0)  circle [radius=0.05]
            (1,1)  circle [radius=0.05]
            (2,0) circle [radius=0.05]
            (3,1) circle [radius=0.05]
            (4,2) circle [radius=0.05]
            (5,3) circle [radius=0.05]
            (6,2) circle [radius=0.05]
            (7,3) circle [radius=0.05]
            (8,2) circle [radius=0.05]
            (9,1) circle [radius=0.05]
            (10,0) circle [radius=0.05]
            ;

	\end{tikzpicture}

\caption{A Dyck path of length $10$ with weight $u(0)^2 v(0)^2 u(1)v(1) u(2)^2 v(2)^2.$}\label{fig:weighted}
\end{figure}

  As observed in \cite{ced}, the phase structure of chase-escape with death is connected to weighted Catalan numbers. We generalize this connection. Let $D_j = \sum_{i=1}^j \rho_i$ and define the weights
  \begin{align}
u(j) =  \f{\lambda_{j+1}}{1+\lambda_{j+1} + D_{j+1} } \text{ and } v(j) = \f{1}{1+\lambda_{j+2} +D_{j+2}} \label{eq:uv}.
\end{align}
 Given a lattice path $\gamma$ consisting of unit rise and fall steps, each rise step from $(x,j)$ to $(x+1,j+1)$ has weight $u(j)$, while a fall step from $(x,j+1)$ to $(x+1,j)$ has weight $v(j)$. The \emph{weight} $\omega(\gamma)$ of a Dyck path $\gamma$ (nonnegative lattice path starting at $(0,0)$ consisting of $k$ rise and $k$ fall steps) is the product of the rise and fall step weights along $\gamma$. See Figure~\ref{fig:weighted}. The corresponding \emph{weighted Catalan number} is 
\begin{align}C_k^{\vlam ,\vrho } = \sum_{\gamma \in \Gamma_k} \omega(\gamma)\label{eq:Ck}
\end{align}
where $\Gamma_k$ is the set of all Dyck paths of length $2k$. Denote the generating function by
\begin{align} 
g(z) &= \sum_{k \geq 0} C_k^{\vlam ,\vrho } z^k. \label{eq:g}
\end{align}
Let $M$ be the largest value such that $|g(z)|< \infty$ for all complex numbers $|z| < M$. 

The weights at \eqref{eq:uv} make it so the $C_k^{\vlam ,\vrho }$ correspond to the probability that a renewal, where blue is once again adjacent to the rightmost red site, occurs in a version of distance-dependent chase-escape on the nonnegative integers (see \eqref{eq:Rcat}). In the setting from \cite{ced}, it is proven that the radius of convergence of $g$ relative to the degree of the tree determines the phase. This continues to hold in distance-dependent chase-escape.

\begin{theorem} \thlabel{thm:main}
   Suppose that there exist constants $c>0$ and $m\geq 1$ such that
   \begin{align} \label{eq:hyp}
    \prod_{i = 3}^{\ell -2} \left( 1 + \frac{\lambda_i}{1 + D_i} \right) \leq c\ell^m \quad \text{ for all $\ell \geq 5$},
    \end{align}
    and that for $u$ and $v$ as at \eqref{eq:uv} 
    \begin{align}
	    \lim_{j \to \infty} u(j) v(j) = 0. \label{eq:hyp2}
    \end{align}
Then for $d \geq 2$, expected coexistence on $\mathcal T_d$ occurs if and only if $M \leq d$ with $M$ the radius of convergence of the generating function defined at \eqref{eq:g}. 
\end{theorem}

The hypotheses \eqref{eq:hyp} and \eqref{eq:hyp2} are consequences of generalizing the main result of \cite{ced}. It is unclear how much \eqref{eq:hyp} can be relaxed. It is fairly robust. For example, \eqref{eq:hyp} and \eqref{eq:hyp2} hold so long as $\lambda_i/D_i = O(i^{-\epsilon})$ for some $\epsilon >0$. The hypothesis at \eqref{eq:hyp2} is essential to our argument as it allows us to apply \thref{thm:worpitzky}. Note that \eqref{eq:hyp} and \eqref{eq:hyp2} are easily verified in the case that $\vlam$ and $\vrho$ are constant. Since we focus on the weaker requirement of expected coexistence, we are able to sidestep many technical difficulties. As a result, we give a streamlined presentation that clarifies and builds on some of the main ideas from \cite{ced}. 

In Section~\ref{sec:Z}, we analyze distance-dependent chase-escape on the non-negative integers. This lets us connect renewal events in the one-dimensional process to weighted Catalan numbers. In \thref{lem:bigger}, we prove that the probability blue reaches beyond a given distance is comparable to the probability a renewal occurs at that distance. Section~\ref{sec:proof} contains the proof of \thref{thm:main}. \thref{lem:bigger} lets us upper bound $\E[|\B|]$ in terms of $g(d)$. This lets us deduce that $d < M$ implies $\E[|\B|] < \infty$. When $\E[|\B|] < \infty$, it is easy to see that $d \leq M$. To handle the boundary case $M=d$, we employ two classical results. One is Worpitsky's Circle Theorem, which we apply via a continued fraction characterization of $g(z)$. We then apply Pringsheim's Theorem to deduce that $\E[|\B|] = \infty$ when $M=d$.

\section{Distance-dependent chase-escape on the integers} \label{sec:Z}

We begin by defining distance-dependent chase-escape on the non-negative integers, which is equivalent to the case $\mathcal T_1$ with $\mathfrak b = 0$ and $\0 =1$. We indicate the state of the vertex $n$ at time $t$ by $s_t (n) \in \{ w, b, r, \dag \}$, which indicates if the vertex is white, blue, red or dead.  Initially, $s_0 (0) =  b $, $s_0 (1) =  r $ and $s_0 (n) = w$ for any $n > 1$, and the process follows the dynamics of distance dependent-chase escape with rates $\vlam$ and $\vrho$. 
For each time $t \geq 0$, we write $B_t = \sup \{ n \colon s_t (n) = b  \} $ and $R_t = \sup \{ n \colon s_t (n) = r \}$. Define the maximum integer reached by $B_t$ as
\begin{align}
    Y &= \sup \{ B_t \colon t \geq 0  \}. \label{eq:Ydef}
\end{align}

As in \cite{ced}, we are mainly interested in times at which the process renews. For each vertex $k \geq 0$, call
\[
    \RR_k = \{ B_t = k, \; R_t = k+1 \text{ and }  s_t (n) \neq \dag  \text{ for all }  n > k+1 \}
\]
a \emph{renewal event at the vertex $k$}. At these points, the process exhibits its initial conditions with a translation by $k$. For $t \geq 0$, define the event $ A_t = \{ s_t (k) \neq \dag \text{ for all } k \}  $ that there are no killed red sites at time $t$.

We follow the evolution of distance-dependent chase-escape with a Markov chain. Let $S_t = R_t - B_t$ be the distance between the rightmost blue and red particles at time $t \geq 0$. We define a discrete version of $(S_t)_{t \geq 0}$ by considering the collection of times where a particle changes its state.  Let $\tau (0) = 0 $ and 
\[
    \tau (i) = \inf \{ t \geq \tau (i - 1) \, : \, S_t \neq S_{\tau (i - 1) } \text{ or } \mathds{1} (A_{t} ) = 0  \}.
\]
The \emph{jump chain} $J = (J_i)_{i \in \mathbb{Z}_+}$ of $(S_t)_{t \geq 0}$ is
defined by 
\[
    J_i = 
    \begin{cases}
    S_{\tau(i)},    & \mathds{1} (A_{\tau(i)}) = 1 \\
    0               & \text{otherwise}
    \end{cases}.
\]
We say that a jump chain is \emph{living at step $n$} if $J_i > 0$ for all $0 \leq i \leq n$.
The transition probabilities of the jump Markov chain are, for each $j > 0$,
\[
    p_{j, j+1} = \frac{\lambda_j}{1 + \lambda_{j} + D_j}, 
    \quad
    p_{j, j-1} = \frac{1}{1 + \lambda_{j} + D_j}, 
    \quad
    p_{j,0} = \frac{D_j}{1 + \lambda_{j} + D_j},
\]
where $D_j = \sum_{i = 1}^j \rho_i $. 

On the event $\RR_k$, the path of the jump chain $J$ (up to the renewal time) can be identified with a Dyck path of length $2k$ translated by one vertical unit. Moreover, the weights that we considered in~\eqref{eq:uv} correspond to the transition probabilities of the jump chain: $p_{j, j+1} = u(j - 1)$ and $p_{j, j-1} = v(j -2)$. It follows that
\begin{align}
    \P_{\vlam, \vrho}(\RR_k) = C_{k}^{\vlam, \vrho }\label{eq:Rcat}
\end{align}
the weighted Catalan number defined at \eqref{eq:Ck} that uses the weights at \eqref{eq:uv}.

We use ideas from \cite[Lemma 2.2]{ced} to prove that $\P(Y \geq k)$ is comparable to $\P(\RR_k)$. The difficulty is that the event $\{Y\geq k\}$ includes \emph{all} realizations for which blue reaches $k$, while $\RR_k$ only includes realizations which have a renewal at $k$. 

\begin{lemma} \thlabel{lem:bigger}
 Let $C$ and $m$ be as in \eqref{eq:hyp}. There exists $c_0$ that does not depend on $k$ such that 
    \[  \P (Y \geq k) \leq c_0 k^{1 + m} \P (\RR_k) \] for all $k \geq 1$.
\end{lemma}

% \begin{enumerate} 
%     \item $\rho_i =0$ for $i \geq m$ and $\lambda_{i+km} = \lambda_i$ for $i \leq m$. 
%     \item $\sum_1^\infty \rho_i = L < \infty $. 
% \end{enumerate}
%     In the former, the critical value can likely be computed so long as $\vlam$ is sufficiently under control.  

\begin{proof}
  %  We follow the strategy in the proof of~\cite[Lemma 2.2]{ced}. 
   % We consider $\Gamma_Y$ and $\Gamma_R$ to be the set of paths in the jump chain associated to realizations of the events $ \{ Y \geq k \} $ and $  \RR_k  $.
    %Note that a jump chain in $\Gamma_Y$ may be killed before a blue particle arrives at $k$, but the process continues on the line. This does not affect our argument, because we follow the jump chain until the arrival of a red particle to $k$.

    Define the \emph{height profile} of a (jump chain) path $J = (J_0, J_1, \ldots, J_m)$ to be the vector $h (J) = (h_1 (J), \ldots, h_{m+1} (J) ) $ which indicates the frequency of each height reached by $J$. Formally, $h_i(J)  \coloneqq \sum_{\ell = 0}^{m} \mathds{1} (J_{\ell} = i) $.
    The probability that the distance-dependent chase-escape follows the path $J$ in its jump chain is
    \[
        p (J) = \prod_{i = 1}^{k} \lambda_i^{h_i(J)} \prod_{j = 1}^{m + 1} \left( \frac{1}{1 + \lambda_j + D_j} \right)^{h_j (J)}.
    \]
    
    It is necessary and sufficient for the occurrence of $\{ Y \geq k \}$ to \emph{first} have the rightmost red particle reach $k$, and \emph{second} have the rightmost blue particle reach $k$ after $\ell$ steps. An advantage of this perspective is that the occurrence of second stage only depends on $\ell$. In particular, it does not depend on the behavior of red beyond $k$. We now formally decompose jump chains corresponding to $\{Y \geq k\}$ into these two stages. 
    
        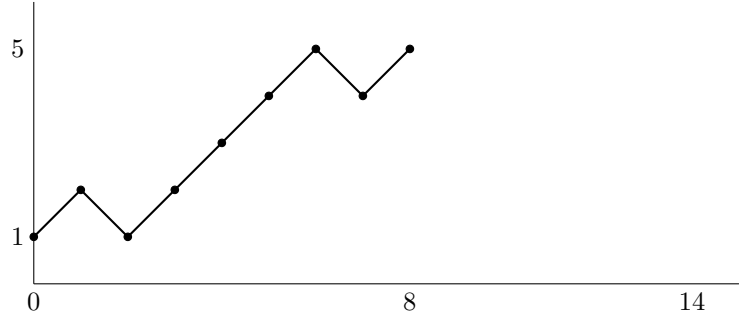
\begin{figure} 
\begin{center}
\begin{tikzpicture}[scale = 1.25]

% horizontal axis
\draw (0,0) -- (7.5,0);
% x-labels
\draw	(0,0)   node[anchor=north] {0}
%		(2,0)   node[anchor=north] {4}
		(4,0)   node[anchor=north] {8}
		(7,0) node[anchor=north] {14};
% y-labels
\draw	(0,.5)  node[anchor=east] {1}
%		(0,1.5)  node[anchor=east] {3}
		(0,2.5) node[anchor=east] {5};
% vertical axis
\draw (0,0) -- (0,3);
% J_i
\draw[thick] (0,.5) -- (.5,1) -- (1,.5) -- (1.5,1) -- (2,1.5) --(2.5, 2) -- (3,2.5) -- (3.5, 2) -- (4,2.5);
\draw[fill] (0,.5)  circle [radius=0.04]
            (.5,1)  circle [radius=0.04]
            (1,.5) circle [radius=0.04]
            (1.5,1) circle [radius=0.04]
            (2,1.5)  circle [radius=0.04]
            (2.5,2) circle [radius=0.04]
            (3,2.5) circle [radius=0.04]
            (3.5,2) circle [radius=0.04]
			(4,2.5) circle [radius=0.04];
\end{tikzpicture}
\caption{
	A path $\gamma \in \Gamma_{5}$ when $k=7$.
} \label{fig:Gl}
\end{center}
\end{figure}

    For a jump chain $\gamma$ corresponding to a configuration from $\{Y \geq k\}$, let  $\ell \in \{ 1, \ldots, k-1 \}$ be the value of the jump chain when red arrives to $k$. 
    %We can follow the first state of the event in the jump chain. 
    % In this case,
    % the path $\gamma$ reaches the position $ (2k - \ell - 1, \ell) $ coming from $ (2k - \ell - 2, \ell-1) $.  \sarai{[In this case, the length (or time) of the jump chain is   $2k - \ell -1$.]} 
    After this, we enter the second stage of $\{ Y \geq k \}$; the rightmost blue particle must advance from $\ell$ to $k$. So, the probability of the second stage is
    \begin{equation}
        \sigma(\ell) =  \prod_{n = 1}^{\ell} 
                        \P \left( \begin{array}{c} 
                        \text{blue advances 1 step before}\\
                                 \text{any site $k - \ell + n, \ldots , k$ dies}  
                                 \end{array} \right)
                 = \prod_{n = 1}^{\ell} \frac{1}{1 + D_n}.
    \end{equation}
    Let $\Gamma_{\ell}$ be the set of all living jump chain paths of length $2k - \ell - 1$ that start at $(0,1)$ and end with an upward step to $(2k - \ell -1, \ell)$. See Figure~\ref{fig:Gl}.

    For $\gamma \in \Gamma_{\ell} $, the probability that  $Y \geq k$ and the first $2k - \ell -1$ steps of the jump chain follow $\gamma$ is
     $ q(\gamma) \coloneqq p(\gamma) \sigma (\ell) $. 
     %\sarai{[for the probability that the jump chain is in $ \Gamma_{\ell} \cap \Gamma_Y $.]}
    Setting $q_{\ell} \coloneqq \sum_{\gamma \in \Gamma_{\ell}} q(\gamma)$, we arrive at the following decomposition of $\{Y \geq k\}$:
    \begin{align}
	    \P ( Y \geq k ) = \sum_{\ell=2}^k \sum_{\gamma \in \Gamma_\ell} p(\gamma) \sigma(\ell) = \sum_{\ell = 2 }^{k} q_{\ell} . \label{eq:Yl}
	\end{align}
        
   A subset of $\RR_k$ is the collection of processes which follow jump chains in $\Gamma_2$, and subsequently have blue advance by one, then red advance by one, followed by blue advancing one. This gives the bound 
    \[
        q_2 \frac{\lambda_2}{(1 + \lambda_1 + D_1) (1 + \lambda_2 + D_2)^2} \leq \P (\RR_k).
    \]
    We will now prove that $q_{\ell}$ for $\ell \geq 2$ is comparable to $q_2$.

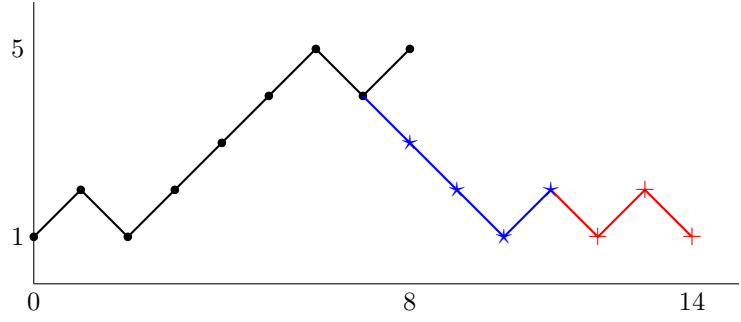
\begin{figure} 
\begin{center}
\begin{tikzpicture}[scale = 1.25]

% horizontal axis
\draw (0,0) -- (7.5,0);
% x-labels
\draw	(0,0)   node[anchor=north] {0}
%		(2,0)   node[anchor=north] {4}
		(4,0)   node[anchor=north] {8}
		(7,0) node[anchor=north] {14};
% y-labels
\draw	(0,.5)  node[anchor=east] {1}
%		(0,1.5)  node[anchor=east] {3}
		(0,2.5) node[anchor=east] {5};
% vertical axis
\draw (0,0) -- (0,3);
% J_i
\draw[thick] (0,.5) -- (.5,1) -- (1,.5) -- (1.5,1) -- (2,1.5) --(2.5, 2) -- (3,2.5) -- (3.5, 2) -- (4,2.5);
\draw[fill] (0,.5)  circle [radius=0.04]
            (.5,1)  circle [radius=0.04]
            (1,.5) circle [radius=0.04]
            (1.5,1) circle [radius=0.04]
            (2,1.5)  circle [radius=0.04]
            (2.5,2) circle [radius=0.04]
            (3,2.5) circle [radius=0.04]
            (3.5,2) circle [radius=0.04]
			(4,2.5) circle [radius=0.04];
%%\tilde \gamma						

\draw[thick, blue]  (3.5,2) -- (5,.5) -- (5.5,1);
%\draw (3.5,2)  circle [radius=0.05]
%            (4,1.5)  circle [radius=0.05]
%            (4.5,1)  circle [radius=0.05]
%            (5,.5)  circle [radius=0.05]
%			(5.5,1)  circle [radius=0.05];

%\draw[thick] (5,.5) -- (5.5,1);
%%tilde gamma
\draw[thick, red]  (5.5,1) -- (6,.5) -- (6.5,1) -- (7,.5);
\node[blue] at (5.5,1) { \Large $\star$};
\node[blue] at (5,.5) {\Large $\star$};
\node[blue] at (4.5,1) {\Large $\star$};
\node[blue] at (4,1.5) {\Large $\star$};
%R_k
\node[red] at (6,.5) {$+$};
\node[red] at (6.5,1) {$+$};
\node[red] at (7,.5) {$+$};
\end{tikzpicture}
\caption{
	Let $k=7$. The black line with dots is a path $\gamma \in \Gamma_{5}$. The blue line with stars is the modified path $\tilde \gamma \in \Gamma_2$. The red line with pluses is the extension of $\tilde \gamma$ to a jump chain in $\RR_7$.  
} \label{fig:path}
\end{center}
\end{figure}

    Given $\gamma \in \Gamma_{\ell}$, we obtain $\tilde{\gamma} \in \Gamma_{2}$ inserting $\ell - 2$ downward steps before the last upward step. See Figure~\ref{fig:path}.
    The paths $\gamma$ and $\tilde{\gamma}$ agree on the first $2k - \ell -2$ steps, so
    \begin{align}
        q (\gamma) &= \frac{\sigma (\ell)}{ \sigma (2) \prod_{i = 1}^{\ell -2} (1 + \lambda_i + D_i)^{-1} } q (\tilde{\gamma}) \\
        &= \f{ \prod_{n = 1}^{\ell} \frac{1}{1 + D_n}}{ \prod_{n = 1}^{2} \frac{1}{1 + D_n}  \prod_{i=1}^{\ell-2}(1 + \lambda_i + D_i)^{-1} }  q (\tilde{\gamma})
    \end{align}
    Simplifying gives
    \begin{align}
                   q(\gamma) &= q (\tilde{\gamma}) \frac{ (1 + D_1 + \lambda_1) (1 + D_2 + \lambda_2) }{ (1 + D_{\ell})(1 + D_{\ell-1}) } 
                   \prod_{i = 3}^{\ell -2} \left( 1 + \frac{\lambda_i}{1 + D_i} \right) . 
    \end{align}
    
    Taking $C$ and $m$ as at \eqref{eq:hyp} set $$c_0 = c \frac{ (1 + D_1 + \lambda_1) (1 + D_2 + \lambda_2) }{ (1 + D_{\ell})(1 + D_{\ell-1}) }.$$
    It follows from our hypothesis \eqref{eq:hyp} that for all $\gamma \in \Gamma_\ell$ 
    $$q(\gamma) \leq c_0 \ell^m q(\tilde \gamma).$$ 
    As the map $\gamma \in \Gamma_\ell \mapsto \tilde \gamma \in \Gamma_2$ is injective, we have
\begin{align}
	q_\ell \leq \sum_{\gamma \in \Gamma_\ell} c_0 \ell^{m} q(\tilde \gamma) \leq c_0 \ell^{m} \sum_{\gamma \in \Gamma_2} q(\gamma) =  c_0 \ell^{m} q_2. \label{eq:ql2}
\end{align}
Applying \eqref{eq:ql2} at \eqref{eq:Yl} and summing completes the lemma.

\end{proof}

\section{Proof of \thref{thm:main}} \label{sec:proof}

\subsection{Preliminaries}
This treatment is similar to that in \cite[Section 3]{ced}. For completeness, we reproduce and generalize the necessary elements. Given a sequence $(c_n)_{n \geq 0}$, define the formal continued fraction
\begin{align}
K[c_0,c_1,\hdots]  := \cfrac{ c_0}{ 1- \cfrac{c_1}{1-  \ddots}}.\label{eq:K}
\end{align}
%The $c_i$ may be fixed numbers, or possibly functions. Also, whenever we write $x$ we mean a nonnegative real number, and $z$ represents an arbitrary complex number. 
Let \begin{align}
a_j := u(j)v(j)	=\f{\lambda_{j+1}}{(1+\lambda_{j+1} + D_{j+1})(1+\lambda_{j+2} +D_{j+2}) }.
\end{align}
It follows from \cite[Chapter 5]{goulden1985combinatorial} that 
\begin{align}g(z) = f(z):= K[1,a_0z,a_1z,\hdots] \label{eq:f}\end{align}
for all $|z| < M$. 
A classical theorem of Worpitzky lets us prove that $f$ is \emph{meromorphic}  i.e.\  holomorphic outside of a set of isolated poles. 

\begin{theorem}[Worpitzky Circle Theorem]\thlabel{thm:worpitzky}
	Let $c_j\colon D \to \{|w| < 1/4\}$ be a family of analytic functions over a domain $D \subseteq \mathbb C$. Then
	$K[1,c_0(z),c_1(z),\hdots]$ 
	converges uniformly for $z$ in any compact subset of $D$.
	%and the limit function takes values in $\{ |z-4/3| \leq 2/3\}$.
\end{theorem}

\begin{corollary} \thlabel{lem:meromorphic}
Assuming \eqref{eq:hyp2}, it holds that $f$ is meromorphic on $\mathbb C$.
\end{corollary}

\begin{proof}
%Let $\Delta = \{ z \colon |z| \leq M\}$. 
We will prove that $f$ is meromorphic for all $z \in \Delta = \{ |z| < r_0\}$ with $r_0>0$ arbitrary.  Let $T_j(z) := K[a_jz, a_{j+1} z, \hdots]$
be the tail of the continued fraction so that $f(z) = K[1,a_0z, \hdots, a_{j-1}z, T_j(z)]$.  By the hypothesis \eqref{eq:hyp2}, we have $|a_j|\downarrow 0$ as $j \to \infty$.
 It follows that for some $j= j(r_0)$ large enough, $|a_k z| \leq 1/4$ for all $k \geq j$ and $z\in \Delta$. \thref{thm:worpitzky} ensures that $|T_j(z)| < \infty$ and the partial continued fractions $K[a_j z, \hdots, a_n z]$
are analytic (again by \thref{thm:worpitzky}) and converge uniformly to $T_j$ for $z \in \Delta$.
 Thus, $T_j$ is a uniform limit of analytic functions and is therefore analytic on $\Delta$. 
 We can then write
$f(z) = K[1,a_0z , \hdots, a_{j-1} z, T_j(z)].$
Since each $a_i z$ is a linear function in $z$, $f$ is a quotient of two analytic functions.
\end{proof}

Our next lemma requires a classical theorem from complex variable theory (see \cite[Theorem IV.6]{flajolet2009analytic} for example). 

\begin{theorem}[Pringsheim's Theorem] \thlabel{thm:pringsheim}
	If $\varphi(z)$ is representable at the origin by a power series $\varphi(z) = \sum_{n=0}^\infty a_n z^n$ that has real coefficients $a_n\geq0$ and radius of convergence $M$, then the point $z = M$ is a singularity of $\varphi(z)$.
\end{theorem}

\begin{lemma}\thlabel{lem:M_regimes}
Let $\rho > 0$. Then
$M\leq d$ if and only if $g(d) = \infty$. 
\end{lemma}

\begin{proof}
We first note that the implication ``$M < d$ implies $g(d) = \infty$" as well as the reverse direction ``$g(d) = \infty$ implies $M \leq d$" both follow immediately from the definition of the radius of convergence. 
It remains to show that $M=d$ implies $g(d) = \infty$.
\thref{lem:meromorphic} proves that $f$ is a meromorphic function. Since $g=f$ for $|z| <M$, $f(x) > 0$ for $x \in (0,d)$ we have 
$$g(d) = \lim_{x\uparrow d} g(x) = \lim_{x\uparrow d} f(x).$$
Moreover, \thref{thm:pringsheim} gives $z=d$ is a singularity. Monotone convergence and nonnnegativity of the coefficients then ensure that $g(d^-) = f(d^-)= \infty$. 
\end{proof}

\subsection{Proof of \thref{thm:main}}

\begin{proof}
Index the $d^k$ vertices at distance $k$ from the root of $\mathcal T_d$ by $(v_{k,i})_{i=1}^{d^k}$. Self-similarity of the tree ensures that
\begin{align}
\E[|\B|] &= 1 + \E \left[ \sum_{k=0}^\infty \sum_{i=1}^{d^k} \ind{\text{$v_{k,i}$  is eventually blue} }\right] \\
&=1+ \sum_{k=0}^\infty \P(Y=k) d^k.\label{eq:Bb}
\end{align}

Now, suppose that $M>d$. Using \eqref{eq:Bb} and the comparison in \thref{lem:bigger} gives
\begin{align}
E [|\B|] &\leq 1+ \sum_{k=1}^ \infty ck^{1+m}\P(\RR_k) d^k\label{eq:EB}.
\end{align}
Since $M>d$, the sum on the right converges even with the polynomial prefactor. Thus, $\E[|\B|] < \infty$.

Next, suppose that $\E[|\B|] < \infty$. Let $Y$ be as at \eqref{eq:Ydef}. 
Since $\{Y=k\}$ contains the event that a renewal occurs at $k$ followed by blue advancing one step, we have $$\P(\R_{k-1}) \f{1}{1+ \lambda_1 + \rho_1}  \leq \P(Y=k).$$ Applying this bound to \eqref{eq:Bb} and reindexing the sum gives
$$\E[|\B|] \geq \f{d}{1+ \lambda_1 + \rho_1}  \sum_{k=0}^\infty \P(\RR_k) d^k = \f{1}{1+ \lambda_1 + \rho_1} g(d).$$
Hence, $g(d)<\infty$, which gives $M >d$ by \thref{lem:M_regimes}. 
\end{proof}

% \section{Theorems to replicate from \cite{ced}}

% Theorem... $E|\B|< \infty $ if and only if $d<   M(\vlam, \vrho)$.

% \begin{itemize}
%     \item Lemma 2.2
%     \item Theorem 3.1
%     \item Corollary 3.2
%     \item Theorem 3.6
%     \item Lemma 3.7
%     \item Lemma 4.2    
% \end{itemize}

\bibliographystyle{amsalpha}
\bibliography{ced}

\end{document}